\documentclass[a4paper,10pt]{amsart}
\usepackage[T1]{fontenc}
\usepackage[utf8]{inputenc}
\newtheorem{thm}{Theorem}

\newtheorem{lemma}{Lemma}

\title{A note on central moments in $C^*$-algebras}
\author{Zoltán Léka}
\address{HAS Alfr\'ed R\'enyi Institute of Mathematics \\ 1053 Budapest \\ Re\'altanoda 13-15}
\email{leka.zoltan@renyi.mta.hu}

\thanks{This study was partially supported by the Hungarian NSRF (OTKA) grant
no. K104206. and the "Lendület" Program (LP2012-46/2012) of the Hungarian
Academy of Sciences }


\date{}

\begin{document}

\begin{abstract}
    We present sharp estimates of the $k^{\rm th}$ central moments of normal elements in $C^*$-algebras. We shall obtain an upper bound for
     the weak moments of general elements as well.
\end{abstract}

\maketitle

\section{Introduction}

Variance and higher-order central moments are well-studied concepts in probability theory and statistics. Our starting point is a well-known lemma that originates from a paper of Murthy and Sethi \cite{MS} and provides an estimate
for the standard variance of real random variables in terms of their largest and smallest values. It says that if $X$ is a real discrete random variable
taking values $x_i$ with probability $p_i$, then
$$ \mbox{Var}(X) = \sum_i p_i x_i^2 - \left(\sum_i p_i x_i \right)^2 \leq {1 \over 4} (M - m)^2 $$ for any constants
$m \leq X \leq M$. Recently, K. Audenaert \cite{A} gave a sharp extension of the result for
complex variables and even for matrices in terms of different types of quantum variances. For instance, he proved
for any $A \in \mathbb{C}^{n \times n}$ that the equality
$$\max \{  \mbox{Tr } D|A- \mbox{Tr } DA|^2 : 0 \leq D \in \mathbb{C}^{n\times n} \mbox{ and }  \mbox{Tr } D = 1\} = \min_{\lambda \in \mathbb{C}} \|A - \lambda I\|^2 $$
 holds, where $\|\cdot\|$  denotes the operator norm (see \cite[Theorem 9]{A}).
 Later R. Bhatia and R. Sharma obtained Audenaert's result applying a
 different method in \cite{BS}. Furthermore, in the setting of $C^*$-algebras the full extension of these variance (or standard deviation) estimates appeared in M. Rieffel's paper \cite{R2}. His approach is based on the proof of a variant of Arveson's distance formula for $C^*$-algebras.
 Quite recently T. Bhattacharyya and P. Grover \cite{BG} presented a new proof of these inequalities exploiting  Birkhoff--James orthogonality.

 Our aim here is to provide sharp inequalities on higher moments of normal elements in $C^*$-algebras. We shall prove that our results are
 sharp for the  $k^{\rm th}$ moments if $k$ is even. The approach that we use is different from the aforementioned ones and rather elementary as well. In the first part of the paper we shall study the (weak) moments of
 Hermitian elements. With these preliminary results at hand we can give the moment estimates for normal elements. In the last section of the paper we shall present
 an inequality for the weak central moments of general $C^*$-algebra elements.

\section{Central moments of Hermitian elements}

 Let $\mathcal{A}$ denote a unital $C^*$-algebra, and let $\mathcal{S}(\mathcal{A})$ stand
for the set of states of $\mathcal{A},$ i.e. the set of positive linear functionals of norm $1$. Obviously, $\mathcal{S}(\mathcal{A})$ is compact in the weak-$*$ topology. The $k$th root of the maximal $k$th central moment of
$a \in \mathcal{A}$ is defined by
  $$ {\Delta}_k(a) = \sup_{\omega \in \mathcal{S}(\mathcal{A})}   \omega(|a - \omega(a)|^k)^{1/k} =  \max_{\omega \in \mathcal{S}(\mathcal{A})} \omega(|a - \omega(a)|^k)^{1/k} . $$
(For any $a \in \mathcal{A},$ $|a| = (a^*a)^{1/2}$ by definition.)

 In \cite[Theorem 3.10]{R2} Rieffel proved for any $a \in \mathcal{A}$ that
                        $$  {\Delta}_2(a) = \min_{\lambda \in \mathbb{C}} \|a-\lambda {\bf 1}\|.$$
  Furthermore, he also observed that the factor norm obeyed the strong Leibniz inequality which led him to prove that the standard deviation is actually a strongly Leibniz seminorm (for the
  details see \cite{R2}).
  L. Molnár \cite{M} proved that if $a$ is self-adjoint, then ${\Delta}_2(a)$ is the same as the factor norm of $a$ in $\mathcal{A}/\mathbb{C}.$
  The proof there exploited a simple geometrical fact to conclude that $${\Delta}_2(a) =  \min_{\lambda \in \mathbb{R}} \|a-\lambda {\bf 1}\| = {1 \over 2} \mbox{diam } \sigma(a),$$ where $\sigma(a)$
  is the spectrum of $a.$

  To get a better understanding of the central moments in $C^*$-algebras, we shall need the following auxiliary concept of moments.

  \bigskip

  \noindent {\bf Definition.} For any $a \in \mathcal{A}$ and $\omega \in S(A),$ let us say that
   $ \omega((a - \omega(a))^k)$ is a {\it weak $k^{\rm th}$ central moment} of
   $a.$

  \bigskip

  Similarly to the notions above, let us introduce the notation
   $$ \mathring{\Delta}_k(a) = \max_{\omega \in \mathcal{S}(\mathcal{A})}  |\omega((a - \omega(a))^k)|^{1/k}.$$
  Now we can give a geometric meaning to the higher-order weak moments of Hermitian elements. From here on we shall denote the maximum norm by
   $$\|p\|_\infty = \max_{0 \leq x \leq 1} |p(x)|$$
  for any polynomial $p$ on the interval $[0,1].$
  The $C^*$-algebra generated by a single element $a$ is denoted by $C^*(a).$

 \begin{thm} Let $a$ be a Hermitian element of a unital $C^*$-algebra $\mathcal{A}$. Then
  $$  \mathring{\Delta}_k(a) = \|p_k\|_\infty^{1/k} \: {\rm diam } \: \sigma(a),$$
 where $p_k(x) = x(1-x)^k + (-1)^{k}x^{k}(1 - x).$ Moreover, equality holds with an $\omega \in \mathcal{S}(C^*(a))$ which
 is the convex sum of two Dirac masses concentrated at the
 farthest spectrum points of $a.$
 \end{thm}
  \begin{proof}
    Let $\xi$ be a point where the polynomial $p_k$ attains its maximum on $[0,1]$.
    Let $\varphi_1$ and $\varphi_2$ be multiplicative functionals of the commutative $C^*(a)$  which satisfy the equalities $\varphi_1(a) = \max \{ \lambda : \lambda \in \sigma(a)\}$
    and $\varphi_2(a) = \min \{ \lambda : \lambda \in \sigma(a)\}$ (see e.g. \cite[Proposition 4.2.2]{Pe}). Set the state $\varphi = \xi \varphi_1 + (1-\xi) \varphi_2$ and pick one of its extensions to the whole $\mathcal{A}$ which is also a state. Then  we obtain
    that
    \begin{eqnarray*}
     \begin{split}
       \varphi((a- \varphi(a))^k) &= \sum_{j=0}^k \binom{k}{j} (-1)^j \varphi(a^{k-j}) \varphi(a)^j \\
                                  &= \sum_{j=0}^k  \sum_{l=0}^j \binom{k}{j} \binom{j}{l} (-1)^j  (1-\xi)^{j-l}\xi^{l+1}\varphi_1(a)^{k-j+l}\varphi_2(a)^{j-l} \\
                                     &\quad + \sum_{j=0}^k  \sum_{l=0}^j \binom{k}{j} \binom{j}{l} (-1)^j  (1-\xi)^{j-l+1}\xi^l \varphi_1(a)^l\varphi_2(a)^{k-l}, \\
                                  & \hspace{-2.3cm} \mbox{and with the substitution } m = j - l, \\
                                  &= \sum_{m=0}^k \xi (1-\xi)^{m} \binom{k}{m} (-1)^{m} \varphi_1(a)^{k-m}\varphi_2(a)^{m}  \sum_{l=0}^{k-m} \binom{k-m}{l} (-1)^l \xi^l \\
                                  &\quad + \sum_{l = 0}^k \sum_{j = l}^k \binom{k}{l}  \binom{k-l}{j-l}  (-1)^j  (1-\xi)^{j-l+1}\xi^l \varphi_1(a)^l\varphi_2(a)^{k-l} \\
                                  &= \xi(1-\xi)^k(\varphi_1(a)-\varphi_2(a))^k + (1-\xi)\xi^k(\varphi_2(a)-\varphi_1(a))^k \\
                                  &= p_k(\xi)(\varphi_1(a)-\varphi_2(a))^k.
      \end{split}
     \end{eqnarray*}
   Therefore we get that $$\mathring{\Delta}_k(a)   \geq \|p_k\|_\infty^{1/k} \mbox{ diam } \sigma(a).$$

   To prove the reverse inequality we shall reduce the problem to the finite dimensional setting. The Gelfand--Naimark--Segal construction tells us that
  $\mathcal{A}$ can be realized as a $\ast$-subalgebra of the full operator algebra of a Hilbert space $\mathcal{H}$. Then the Hermitian $a$ has the spectral decomposition
   $$ a = \int_{\sigma(a)}  \lambda \: dp(\lambda),$$ where
   $dp(\lambda)$ denotes the spectral distribution of $a.$ Let us choose a sequence of simple functions $s_n$ which uniformly approximates the
   identity on the spectrum $\sigma(a) \subseteq \mathbb{R}.$ Now fix an $n$ and pick  $s_n = \sum_{i=1}^m \gamma_i \chi_{E_i}$ $(\gamma_i \in \mathbb{R}),$ where  $E_i$-s are disjoint Borel sets
   such that $\displaystyle \cup_{i=1}^m E_i = \sigma(a).$ Set $$a_n =  \int_{\sigma(a)}  s_n \: dp = \sum_{i=1}^m \gamma_i p_{E_i}.$$
   Then for any state $\omega$ on $\mathcal{B}(\mathcal{H}),$ we have
   \begin{eqnarray*}
    \begin{split}
      \omega((a_n - \omega(a_n))^k)  &= \omega \Bigl( (\sum_{i=1}^m (\gamma_i- \omega(a_n)) p_{E_i} ) ^k \Bigr) \\
                                    &= \sum_{i=1}^m (\gamma_i- \omega(a_n)) ^k \omega (p_{E_i})  \\
                                    &= \mathbb{E}((X - \mathbb{E}(X))^k) =: \mu_k^k(X),
    \end{split}
   \end{eqnarray*}
   where $X$ is a discrete random variable on $\{1, \hdots, m \}$ which takes the value $\gamma_i$ with probability $\omega(p_{E_i})$ $(1 \leq i \leq m).$

   We claim that if the $k$th central moment of $X$ is maximal then its distribution must be concentrated on at most two atoms. To see this, from the shift invariance property of $\mu_k$,
   one can simply assume that $\mathbb{E}(X) = 0.$ Hence $\mu_k^k(X)$ is less than or equal to the maximum
   of the linear function of the variables $t_i$ $(1 \leq i \leq m)$
   \begin{eqnarray*}
    \begin{split}
      &\sum_{i=1}^m \gamma_i^k t_i \;
    \mbox{ subject to the linear constraints } \\
    &\sum_{i=1}^m \gamma_i t_i = 0, \quad \sum_{i=1}^m t_i = 1 \mbox{ and } t_1, ..., t_m \geq 0.
    \end{split}
   \end{eqnarray*}
    From the Krein--Milman theorem, it is simple that we can optimize the above problem at the vertices of the convex polytope $\{ (t_1, ..., t_m) \in \Delta^{m-1} : \sum_{i=1}^m \gamma_i t_i = 0  \} \subseteq \mathbb{R}^m,$
    where $\Delta^{m-1}$ denotes the convex hull of the standard basis vectors of $\mathbb{R}^m$ (i.e. $\Delta^{m-1}$ is the standard
    $m-1$-simplex), \cite[Chapter II, (3.4) Corollary]{Ba}. One can instantly see that these vertices, which lie at the intersection of $\Delta^{m-1}$ and an
    affine hyperplane of $\mathbb{R}^m,$ must be on the edges of $\Delta^{m-1}.$ In fact, any vertex of the intersection is the convex
    combination of at most two extreme points of $\Delta^{m-1},$ see \cite[Chapter III, (9.1) Lemma]{Ba}. That is, if $\mu_k$ is maximal then the distribution
     $ (\omega(p_{E_1}), ..., \omega(p_{E_m}))$ must be concentrated on at most two atoms, hence the claim follows.
     
     Now let us say that $\omega(p_{E_i})$ and $\omega(p_{E_j}) = 1 - \omega(p_{E_i})$ are positive (or at least one of them is positive, and $\omega(p_{E_k})= 0$ if $k \neq i,j$ ).
     Since $(\gamma_i - \gamma_j)^{-1}\mu_k(X) =  \mu_k((\gamma_i - \gamma_j)^{-1}(X-\gamma_j)),$ a simple
     computation gives that
     $$  \mu_k^k(X) = (\gamma_i - \gamma_j)^k( \omega(p_{E_i}) (1- \omega(p_{E_i}))^k + (-1)^k\omega(p_{E_i})^k(1-\omega(p_{E_i}))),$$
     whenever $\gamma_i \geq \gamma_j$ holds. We get
       $$  \mu_k^k(X) \leq |\gamma_i - \gamma_j|^k \max_{x \in [0,1]} |p_k(x)|.$$ 
     This also means that the inequality
      $$\omega((a_n - \omega(a_n))^k)^{1/k}  \leq \|p_k\|_\infty^{1/k} \mbox{ diam } \sigma(a_n)$$ clearly holds. Since $a_n \rightarrow a$ in norm, $\mbox{diam } \sigma(a_n) = \mbox{diam ran } s_n \rightarrow \mbox{diam } \sigma(a)$ $(n \rightarrow \infty),$  and $\omega$ is continuous, we conclude that
       $$\mathring{\Delta}_k(a)  \leq \|p_k\|_\infty^{1/k} \mbox{ diam } \sigma(a).$$

     From the construction in the first part of the proof, we get the existence of a convex combination of two Dirac masses concentrated at the largest and the smallest spectrum
     points of $a$ such that equality holds.
     \end{proof}

 {\noindent \bf Remark.} We note that a minor modification of the second part of the proof immediately gives an upper bound on the $k^{\rm th}$ central moment of any
 Hermitian element (for odd $k$). In fact, if $\mathbb{E}(|X - \mathbb{E}(X)|^k)$ is maximal for some real discrete $X$ then its distribution is concentrated on the
 largest and the smallest values of $X.$ As a corollary one can readily prove the inequality
  \begin{equation}
    \Delta_k(a)  \leq  \|q_{k}\|_\infty^{1/k} \: {\rm diam } \: \sigma(a) = 2 \|q_{k}\|_\infty^{1/k}\min_{\lambda \in \mathbb{C}} \|a-\lambda {\bf 1}\|,
  \end{equation}
  where $q_k(x) = x(1-x)^k + x^k(1-x).$ We leave the details to the interested reader.

   \section{Higher-order moments of normal elements}

  Relying on the results of the previous section we can prove some estimates on the central moments of normal elements as well. The
  estimates are sharp when $k$ is even. For the upper bound, we have the following theorem.

 \begin{thm}
   Let $a$ be a normal element of a unital $C^*$-algebra $\mathcal{A}.$ Then
    $$ {\Delta}_k(a)   \leq  2\|q_{k}\|_\infty^{1/k}\min_{\lambda \in \mathbb{C}} \|a-\lambda {\bf 1}\|,$$
   where $q_{k}(x) = x(1-x)^{k} + x^{k}(1 - x).$
 \end{thm}

 \begin{proof}
   Let $\omega_0 \in \mathcal{S}(\mathcal{A})$ be such that
   $$ {\Delta}_k(a) = \omega_0(|a-\omega_0(a)|^k)^{1/k}.$$
   Let $(\mathcal{H}, \pi, v)$ be the GNS representation of $\omega_0$ restricted to $C^*(a),$ where $\pi$ is contractive and $\mathcal{H}$ is
   separable. Pick an $\varepsilon > 0.$
   By the Weyl--von Neumann--Berg theorem \cite[Theorem 39.4]{C}, there exists a diagonal operator $D$ and a compact perturbation $K$ such that
   $\pi(a) = D+K$ and $\|K\| \leq \varepsilon.$
   Let $r_0(a)$ denote the radius of the smallest circle that contains $\sigma(a).$
   Then $r_0(a) = \min_{\lambda \in \mathbb{C}} \|a-\lambda {\bf 1}\|$ because $a$ is normal.
   We recall that the spectrum as a set function is upper semicontinuous in the usual Hausdorff distance (see e.g. \cite[Theorem 3.4.2]{Au}).
   Moreover, ${\Delta}_k(a) = {\Delta}_k(a + \rm{const} \cdot {\bf 1})$ and $r_0(\pi(a)) \leq r_0(a)$ hold, 
   therefore we can clearly assume that $D = \sum_i \lambda_i P_i,$ where $P_i$-s are
   orthogonal projections, and $\lambda_i \in (r_0(a) + \varepsilon) \mathbb{D}.$
   Now set the diagonal elements
    $$ \tilde{D} := \sum_i (\lambda_i P_i \oplus -{\overline \lambda_i} P_i) \quad \mbox{and} \quad \breve{D} := \sum_i |\lambda_i| (P_i \oplus - P_i)$$ in
   $\mathcal{B}(\mathcal{H} \oplus \mathcal{H}).$
   Obviously, $r_0(\breve{D}) = r_0(\tilde{D}) = r_0(D)$ and $\sigma(D) \subseteq \sigma(\tilde{D}).$ From the Krein--Milman theorem, we know that any state of $C^*(D)$ (and that of $C^*(\tilde{D})$) is in the weak-$*$ closure of the
   convex hull of the Dirac masses on $\sigma(D)$ (and on $\sigma(\tilde{D}),$ respectively), see e.g. \cite[Proposition 2.5.7]{Pe}.  Thus we have by means of the Gelfand transform that
    $$ \Delta_k(D) \leq \Delta_k(\tilde{D}).$$
   For an $\varepsilon' > 0$ we obtain that
    \begin{eqnarray*}
     \begin{split}
         \Delta_k(a)  &\leq (1+\varepsilon') \Delta_k(D)   \\
        &\leq (1+\varepsilon') \Delta_k(\tilde{D}) \\
        &= (1+\varepsilon') \max_{\omega  \in \mathcal{S}(\mathcal{B}(\mathcal{H}))} \: \omega\left(\sum_i |\lambda_i|^k |(P_i \oplus -P_i) - \omega(P_i \oplus -P_i)|^{k}\right)^{1/k} \\
        &= (1+\varepsilon') \Delta_k(\breve{D}), \\
   &\hspace{-1.5cm} \mbox{ and by Theorem 1,}     \\
        &= 2(1+\varepsilon') \|q_{k}\|_\infty^{1/k}\min_{\lambda \in \mathbb{C}} \|\breve{D}-\lambda {\bf 1}\| \\
        &= 2(1+\varepsilon') \|q_{k}\|_\infty^{1/k}\min_{\lambda \in \mathbb{C}} \|D-\lambda {\bf 1}\| \\
        & \leq 2(1+\varepsilon')(1+\varepsilon) \|q_{k}\|_\infty^{1/k}\min_{\lambda \in \mathbb{C}} \|a-\lambda {\bf 1}\|.
     \end{split}
    \end{eqnarray*}
   Note that $\varepsilon' \rightarrow 0$ if $\varepsilon \rightarrow 0,$ so the proof is complete.
 \end{proof}

 To prove a lower bound on $\Delta_k(a)$, we shall need a preliminary lemma.
 We recall that the direct sum $\mathcal{A} \oplus \mathcal{A}$ is an ordinary Banach $*$-algebra which is a $C^*$-algebra with the norm
 $$ \|a\oplus b\| = \max \{ \|a\|, \|b\|\}.$$

 \begin{lemma}
   For any normal $a$ in  $\mathcal{A},$
    $$ \Delta_k(a) = \Delta_k(a\oplus a^*).$$
 \end{lemma}

 \begin{proof}
   Obviously, for an $\omega \in \mathcal{S}(\mathcal{A}),$ the functionals $\omega \oplus 0$ and $0 \oplus \omega$ are states on $\mathcal{A} \oplus \mathcal{A},$
   and hence $ \Delta_k(a) \leq \Delta_k(a\oplus a^*).$
   Conversely, let $(\mathcal{H}, \pi)$ be the GNS representation of the $C^*$-algebra $\mathcal{A} \oplus \mathcal{A}$. Fix an $\omega \in \mathcal{S}(\mathcal{A} \oplus \mathcal{A})$
   where $\Delta_k(a\oplus a^*) =  \omega(|a\oplus a^* - \omega(a\oplus a^*)|^k)^{1/k}.$ Since $\Delta_k(a\oplus a^*) = \Delta_k(c(a\oplus a^*))$ for any complex $c$ with modulus $1,$
   there is no loss of generality in assuming that $\omega(a\oplus a^*) \in \mathbb{R}.$
   Let $v \in \mathcal{H}$ be the unit vector for which
   $\omega(\cdot ) = \langle \pi(\cdot)v,v\rangle.$ The ranges of the
   projections $\pi({\bf 1} \oplus 0)$ and $\pi(0 \oplus {\bf 1})$ are orthogonal, hence
   one can write $v$ as the direct sum $x \oplus y,$ where $x = \pi({\bf 1} \oplus 0)v$ and $ y= \pi(0 \oplus {\bf 1})v.$  Furthermore,
    \begin{eqnarray*}
     \begin{split}
     \omega(a\oplus a^*) &= \langle \pi(a \oplus a^*)v, v \rangle \\
                         &=  \langle (\pi(a\oplus 0)\oplus\pi(0 \oplus a^*))(x\oplus y), x\oplus y \rangle \\
                         &= \|x\|^2 \langle \pi(a\oplus 0)x/\|x\|, x/\|x\| \rangle + \|y\|^2 \langle \pi(0 \oplus a)^*y/\|y\|, y/\|y\| \rangle \\
                         &= \|x\|^2 \omega_1(a)  + (1-\|x\|^2) \overline{\omega_2(a)} \\
                         &= \omega_3(a)
     \end{split}
    \end{eqnarray*}
    holds, where $\omega_1, \omega_2$ and $\omega_3 \in \mathcal{S}(\mathcal{A}),$ because $\mathcal{S}(\mathcal{A})$ is convex and closed under conjugation.
    Then
     \begin{eqnarray*}
     \begin{split}
     \omega(|a\oplus a^*-\omega(a\oplus a^*)|^k) &= \omega(|a - \omega_3(a)|^k \oplus |a^*-\omega_3(a)|^k) \\
                         &=  \omega(|a - \omega_3(a)|^k \oplus |(a-\omega_3(a))^*|^k) \\
                         &= \omega_3(|a - \omega_3(a)|^k),
     \end{split}
    \end{eqnarray*}
   which is exactly what we need to get $\Delta_k(a\oplus a^*) \leq \Delta_k(a).$
 \end{proof}

   \begin{thm}
    Let $a$ be a normal element of a unital $C^*$-algebra $\mathcal{A}.$ Then
     $$   2\|p_{k}\|_\infty^{1/k}\min_{\lambda \in \mathbb{C}} \|a-\lambda {\bf 1}\| \leq  \Delta_k(a) .$$
 \end{thm}

 \begin{proof}
    First, pick a $\lambda_0$ such that $\min_{\lambda \in \mathbb{C}} \|a-\lambda {\bf 1}\| =  \|a-\lambda_0 {\bf 1}\|.$
    Let $$\breve{a} := {a-\lambda_0 {\bf 1} \over \|a-\lambda_0 {\bf 1}\|}. $$ Clearly, $\min_{\lambda \in \mathbb{C}} \|\breve{a} - \lambda {\bf 1}\| = \|\breve{a}\|=1.$
    By rotation around $0,$ we can assume that ${\rm i} \in \sigma(\breve{a}).$
    Since $\sigma(\breve{a}\oplus \breve{a}^*) = \sigma(\breve{a}) \cup \overline{\sigma(\breve{a})},$ the spectral mapping theorem implies that
    $$  \min_{\lambda \in \mathbb{C}} \||{\rm i}(\breve{a}\oplus \breve{a}^*)+{\bf 1}| - \lambda{ \bf 1} \| = {1\over 2} \mbox{diam } \sigma(\breve{a}\oplus \breve{a}^*) = 1.$$
    Next, let $\delta_x$ denote the Dirac mass of $C^*({\rm i}\breve{a}\oplus {\rm i}\breve{a}^*+{\bf 1})$ concentrated at $x$ for some $x \in \sigma({\rm i}\breve{a}\oplus {\rm i}\breve{a}^*+{\bf 1}).$
     From Theorem 1, we know that $\mathring{\Delta}_k(|{\rm i}\breve{a}\oplus {\rm i}\breve{a}^*+{\bf 1}|)$ attains its value on the convex sums of $\delta_0$ and $\delta_2$
    (and their extensions to $\mathcal{A}\oplus \mathcal{A}$).
    On the other hand, for any state $\omega_0 = s \delta_0 + (1-s )\delta_2,$ $0 \leq s \leq 1,$ we have via the Gelfand transform that
    $$ \omega_0 (|{\rm i}\breve{a}\oplus {\rm i}\breve{a}^*+{\bf 1}|) =  \omega_0 ({\rm i}\breve{a}\oplus {\rm i}\breve{a}^*+{\bf 1}) = 2(1-s) \geq 0. $$
    The  inequality $(|f|-c)^k \leq |f-c|^k$ simply holds for any continuous function $f$ and any positive constant function $c$ on the Gelfand space of $C^*({\rm i}\breve{a}\oplus {\rm i}\breve{a}^*+{\bf 1}).$
    Hence we obtain
    \begin{eqnarray*}
     \begin{split}
        2\|p_{k}\|_\infty^{1/k}  \|\breve{a}\| &=  2\|p_{k}\|_\infty^{1/k} \min_{\lambda \in \mathbb{C}} \||{\rm i}\breve{a}\oplus {\rm i}\breve{a}^*+{\bf 1}| - \lambda{ \bf 1} \| \\
       &= \max_{\omega_0 =  s \delta_0 + (1-s) \delta_2} \omega_0((|{\rm i}\breve{a}\oplus {\rm i}\breve{a}^*+{\bf 1}| - \omega_0({\rm i}\breve{a}\oplus{\rm i} \breve{a}^*+{\bf 1}))^k)^{1/k} \\
       &\leq \max_{\omega \in \mathcal{S}(\mathcal{A}\oplus \mathcal{A})}    \omega(|{\rm i}\breve{a}\oplus {\rm i}\breve{a}^*+{\bf 1} - \omega({\rm i}\breve{a}\oplus {\rm i}\breve{a}^*+{\bf 1})|^k)^{1/k} \\
       &= \Delta_k(\breve{a}\oplus\breve{a}^*), \\
     \end{split}
    \end{eqnarray*}
    where we used Theorem 1 in the second equality. Finally, an application of the previous lemma completes the proof.
 \end{proof}

  The estimate for the standard variance of normal matrices in the full matrix algebra $M_n(\mathbb{C})$ first appeared in \cite[Theorem 8]{A}.
  From the above theorems we arrive at the main result of the section.

 \begin{thm}
   For any normal element $a$ in a unital $C^*$-algebra $\mathcal{A}$ and $k \in 2\mathbb{N},$
             $$  \Delta_k(a) =   2\|p_{k}\|_\infty^{1/k}\min_{\lambda \in \mathbb{C}} \|a-\lambda {\bf 1}\| .$$
 \end{thm}

 Obviously, the equality $\mathring{\Delta}_k(a) = \Delta_k(a)$ follows for self-adjoint $a$ and even $k.$
 However, one can construct an example which shows that $$\mathring{\Delta}_k(a) < \Delta_k(a)$$ may happen if $a$ is normal and $k$ is even.
 
 \bigskip
 
 {\noindent \bf Example 1.} Let $k=2$ and define the normal matrix
  $$ A = \begin{pmatrix}
                1 & 0 & 0 \cr
                0 & e^{2i\pi/3} & 0 \cr
                0 & 0 & e^{4i\pi/3} \cr
              \end{pmatrix}.
  $$
  From the previous theorem in $M_3(\mathbb{C}),$ we get
  $$  \Delta_2(A) =   2\|p_{2}\|_\infty^{1/2} =  1.$$
  On the other hand,
              $$ \mathring{\Delta}_2(A) \leq \max_{X \colon {\tiny \mbox{ran }} X \subseteq \sigma(A)} \: |\mathbb{E}(X^2) - \mathbb{E}(X)^2|^{1/2}.$$
  Some calculations give that the function $(p_1, p_2, p_3) \mapsto |\mathbb{E}(X^2) - \mathbb{E}(X)^2|$ defined on the probability simplex attains its maximum when
  $X$ is Bernoulli distributed with parameter $1/2.$  Namely,
   $$ \mathring{\Delta}_2(A) \leq {\sqrt{3} \over 2}.$$

\bigskip

  \noindent {\bf Example 2.} If $k$ is odd then it is straightforward to see that Theorem 4 does not hold. In fact, let $k = 3$ and consider
       $$ A = \begin{pmatrix}
                1 & 0 & 0 \cr
                0 & i & 0 \cr
                0 & 0 & 0 \cr
              \end{pmatrix}.
  $$
  Set the state $\omega(X)  = {1 \over 3} \mbox{Tr } X$ on $M_3(\mathbb{C}).$ Then
   $$ \Delta_3(A) \geq |\omega((A-\omega(A))^3)|^{1/3} =
           {\sqrt[3]{50} \over 3\sqrt{2}} \approx 0.868,$$
  while $\displaystyle 2\|p_3(x)\|_\infty^{1/3}\min_{\lambda \in \mathbb{C}} \|A-\lambda I\| = {\sqrt{2} \over \sqrt[6]{108}} \approx 0.648.$

 \section{Notes on general elements}

 In this last section we make few remarks on the central moments of general elements.
 It would be interesting to know whether Rieffel's theorem remains true for larger even moments. Does the equality
 $$ \Delta_k(a) = 2\|q_{k}\|_\infty^{1/k}\min_{\lambda \in \mathbb{C}} \|a-\lambda {\bf 1}\| $$ follow for $k \in 2\mathbb{N},$ where $\|q_{k}\|_\infty$
 denotes the largest $k^{\rm th}$ central moments of the Bernoulli distribution?
 We do not even know the answer in the finite dimensional case, or when $\mathcal{A} = M_n(\mathbb{C}).$
 In the noncommutative setting, we can prove the following weaker statement relying on the von Neumann inequality.

  \begin{thm} Let $\mathcal{A}$ be a unital $C^*$-algebra and let a $\in \mathcal{A}.$ Then for any $k,$ we have
  $$ \mathring{\Delta}_k(a)  \leq  2\|q_{k}\|_\infty^{1/k}\min_{\lambda \in \mathbb{C}} \|a-\lambda {\bf 1}\|,$$
 where $q_{k}(x) = x(1-x)^{k} + x^{k}(1 - x).$
 \end{thm}

 \begin{proof}
    Choose $\lambda_0$ such that $\min_{\lambda \in \mathbb{C}} \|a-\lambda {\bf 1}\| = \|a-\lambda_0 {\bf 1}\|$ holds. Define
     $$\tilde{a} := {a - \lambda_0 {\bf 1} \over \|a - \lambda_0 \bf{1}\|}.$$ Then von Neumann's inequality shows that
       $$ \|p(\tilde{a})\| \leq \sup_{z \in \mathbb{D}} |p(z)| =  \|p(U)\|,$$ for any polynomial $p,$ where $U$ stands for the bilateral shift on $\ell^2(\mathbb{Z}).$ (Recall that $\sigma(U)$ is the unit circle.)
       We note that any state $\omega$ on $\mathcal{A}$ gives rise to a state $\tilde{\omega}$ on $ \mathcal{B}(\ell^2(\mathbb{Z}))$ such that $\omega(p(\tilde{a})) = \tilde{\omega}(p(u))$
       holds for any polynomial $p$ (see \cite[Proposition 33.10]{C}). Thus $\mathring{\Delta}_k(\tilde{a}) \leq  \mathring{\Delta}_k(U).$
       Moreover, from Theorem 2 we get
    \begin{eqnarray*}
     \begin{split}
       \mathring{\Delta}_k(\tilde{a}) \leq  \mathring{\Delta}_k(U) \leq   2 \|q_k\|_\infty^{1/k} \min_{\lambda \in \mathbb{C}} \|U - \lambda I\| =  2 \|q_k\|_\infty^{1/k}
     \end{split}
    \end{eqnarray*}
    which is what we intended to show.
  \end{proof}

  Finally, our next example shows that the above inequality can be sharp for special elements, at least if $k$ is even.

  \bigskip

  \noindent {\bf Example 3.} Let us assume that $S$ is a non-invertible isometry on a Hilbert space $\mathcal{H}.$ Then we can easily see that
  $  2\|p_k\|_\infty^{1/k}\min_{\lambda \in \mathbb{C}} \|S-\lambda I\|\leq \mathring{\Delta}(S) ,$
   where $p_{k}(x) = x(1-x)^{k} + (-1)^kx^{k}(1 - x),$ as usual. Consider the Hermitian matrix
   $$ A = \begin{pmatrix}
            1 & 0 \cr
            0 & -1
          \end{pmatrix} \in M_2(\mathbb{C}).
$$
 Obviously, $\sigma(S) = \overline{\mathbb{D}}.$ The spectral mapping theorem gives that $\sigma(p(A)) \subseteq \sigma(p(S))$ hence $\|p(A)\| \leq \|p(S)\|$ follows for any polynomial $p$ (see e.g. \cite[Theorem 3.2.6]{Au}).
 This means that any state $\omega$ of $M_2(\mathbb{C})$ defines a state $\tilde{\omega}$ on $\mathcal{B}(\mathcal{H})$ such that $\omega(p(S)) = \tilde{\omega}(p(A))$
       holds for any $p$ (\cite[Proposition 33.10]{C}). Therefore,  $\mathring{\Delta}_k(A) \leq   \mathring{\Delta}_k(S).$ Combining this inequality with Theorem 1, we conclude that
    \begin{eqnarray*}
     \begin{split}
        {2} \|p_k\|_\infty^{1/k} \min_{\lambda \in \mathbb{C}} \|S - \lambda I\| &=  {2} \|p_k\|_\infty^{1/k}\min_{\lambda \in \mathbb{C}} \|A - \lambda I\| \\
          &=   \|p_k\|_\infty^{1/k} \mbox{ diam } \sigma(A) \\
          &=   \mathring{\Delta}_k(A)  \\
          &\leq   \mathring{\Delta}_k(S).
     \end{split}
    \end{eqnarray*}
    Now with Theorem 5 at hand we obtain for any non-invertible isometry $S$ and $k \in 2\mathbb{N}$ that $$\mathring{\Delta}_k(S) = {2} \|q_k\|_\infty^{1/k}.$$  
  We point out that Rieffel's theorem (\cite[Theorem 3.10]{R2}, \cite[Corollary 4.2]{BG}) gives $\Delta_2(S) = \mathring{\Delta}_2(S)$ in this special case (compare it with Example 1). Notice that $S$ is Birkhoff--James orthogonal to
  $I;$ i.e. $\|S\| = \min_{\lambda \in \mathbb{C}} \|S - \lambda I \|,$ hence there exists a state $\varphi$ on $\mathcal{B}(\mathcal{H})$ such that $\varphi(I) = 1$ and $\varphi(S) = 0$ (\cite[Corollary 3.3]{BG}).
  Actually with this state we have the equality
    $$  \varphi(|S - \varphi(S)|^2)^{1/2} = \Delta_2(S) = 1.$$
  But if $4 \leq k \in 2\mathbb{N},$ the state $\varphi$ determined by the Birkhoff--James orthogonality does not relate to the higher-order moments because $\mathring{\Delta}_k(S) > 1$ while
  $ \varphi(|S - \varphi(S)|^k)^{1/k} = 1.$

 \section{Acknowledgement}

  I would like to thank Prof. Lajos Molnár for stimulating discussions and questions on the higher-order moments in $C^*$-algebras. I would like to also thank
  the referee for his/her careful reading of the original manuscript and his/her suggestions which made the paper more transparent.

 \end{document}